\documentclass[twoside,a4paper,12pt,centertags]{amsart}
\usepackage{amsmath,amssymb,verbatim,vmargin}
\usepackage[bookmarks=true]{hyperref}   
\usepackage{color}

\theoremstyle{plain}
\newtheorem{thm}{Theorem}[section]
\newtheorem{lem}[thm]{Lemma}

\newtheorem{prop}[thm]{Proposition}

\newtheorem{alg}[thm]{Algorithm}

\theoremstyle{definition}
\newtheorem{defn}[thm]{Definition}

\newtheorem{ex}[thm]{Example}

\mathchardef\semic="303B
\newcommand{\dirac}{{\mathbf D}}

\newcommand{\wedg}{\mathbin{\scriptstyle{\wedge}}}

\newcommand{\inv}[1]{{\widehat{#1}}}
\newcommand{\rev}[1]{\overline{#1}}
\newcommand{\R}{{\mathbf R}}
\newcommand{\C}{{\mathbf C}}

\newcommand{\mH}{{\mathcal H}}

\newcommand{\im}{\text{{\rm Im}}\,}

\newcommand{\clos}[1]{\overline{#1}}
\newcommand{\conj}[1]{\overline{#1}}

\newcommand{\barint}{\mbox{$ave \int$}}

\newcommand{\ta}{{\scriptscriptstyle \parallel}}
\newcommand{\no}{{\scriptscriptstyle\perp}}
\newcommand{\pd}{\partial}

\newcommand{\lap}{\Delta}
\newcommand{\pv}{\text{{\rm p.v.\!}}}

\newcommand{\dl}{{\mathcal Dl}}

\def\barint_#1{\mathchoice
            {\mathop{\vrule width 6pt
height 3 pt depth -2.5pt
                    \kern -8.8pt
\intop}\nolimits_{#1}}%
            {\mathop{\vrule width 5pt height
3 pt depth -2.6pt
                    \kern -6.5pt
\intop}\nolimits_{#1}}%
            {\mathop{\vrule width 5pt height
3 pt depth -2.6pt
                    \kern -6pt
\intop}\nolimits_{#1}}%
            {\mathop{\vrule width 5pt height
3 pt depth -2.6pt
          \kern -6pt \intop}\nolimits_{#1}}}

\usepackage{color}

\definecolor{gr}{rgb}   {0.,   0.8,   0. }
\definecolor{bl}{rgb}   {0.,   0.5,   1. }
\definecolor{mg}{rgb}   {0.7,  0.,    0.7}

\newcommand{\Bk}{}
\newcommand{\Rd}{}

\makeatletter
\@namedef{subjclassname@2010}{%
  \textup{2010} Mathematics Subject Classification}
\makeatother

\begin{document}

\title[A spin integral equation for electromagnetic scattering]
{A spin integral equation for electromagnetic and acoustic scattering}
\author[Andreas Ros\'en]{Andreas Ros\'en$\,^1$}
\thanks{$^1\,$Formerly Andreas Axelsson}
\address{Andreas Ros\'en\\Mathematical Sciences, Chalmers University of Technology and University of Gothenburg\\
SE-412 96 G{\"o}teborg, Sweden}
\email{andreas.rosen@chalmers.se}

\begin{abstract}
We present a new integral equation for solving the Maxwell scattering problem against a perfect conductor. 
The very same algorithm also applies to sound-soft as well as sound-hard Helmholtz scattering, and in fact 
the latter two can be solved in parallel in three dimensions.
Our integral equation does not break down at interior spurious resonances,
and uses spaces of functions without any algebraic or differential constraints.
The operator to invert at the boundary involves a singular integral operator closely related to the three dimensional Cauchy singular integral, 
and is bounded on natural function spaces and depend analytically on the wave number.
Our operators act on functions with pairs of complex two by two matrices as values, 
using a spin representation of the fields.
\end{abstract}


\maketitle

%
%
%
\section{Introduction}


In this paper, we present a new singular integral equation which reads
\begin{equation}   \label{eq:rotintro}
  h(x) -2M(x)\pv\int_{\partial\Omega} \Psi_k(y-x) h(y) dy= \phi(x), \qquad x\in\partial\Omega,
\end{equation} 
and solves the Maxwell as well as the Helmholtz scattering problem, with wave number $k$, in the unbounded connected complement  $\Omega \subset\R^3$ of a compact set, with prescribed boundary data $\phi$ on $\partial\Omega$, which is assumed
to be at least Lipschitz regular.
All the functions $\phi$, $h$ and $\Psi_k$, the latter which essentially is the gradient of the fundamental 
solution of the fundamental solution $\Phi_k$ to Helmholtz equation,
take values in the eight dimensional space of pairs of 
complex $2\times 2$ matrices.
The multiplier $M(x)$ denotes an explicit linear map of such matrix-pairs determined by the 
unit normal vector field $n(x)$.
Depending on whether one wants to solve the Maxwell, Dirichlet Helmholtz or 
Neumann Helmholtz problem, one embeds the boundary data differently in the matrix-pair 
valued function $\phi$, and after having solved the integral equation for $h$ and computed 
the field $F(x)$ with the corresponding integral for $x\in\Omega$, one can extract the 
electric and magnetic fields $E$ and $H$, or the acoustic wave $u$ and its gradient $\nabla u$, 
from this matrix-pair valued function
$F$.

The derivation of the integral equation \eqref{eq:rotintro}, which we refer to as the spin integral equation, 
use a framework for solving 
Dirac transmission problems on Lipschitz domains, from the author's thesis~\cite{Ax},
published in the four papers \cite{Ax1, Ax2, AMc, Ax3}.
This build on earlier works by McIntosh and Mitrea~\cite{McMi} and
Axelsson, Grognard, Hogan and McIntosh~\cite{AGHMc}, where Maxwell's
equations were studied using Clifford analysis. 
Indeed, the above matrix-pairs are nothing but a concrete matrix representation of the 
complex three dimensional Clifford algebra.
The idea to write Maxwell's equation as what we now call a Dirac equation using Clifford algebra, is old and
dates back to the original works of Maxwell and Clifford, and was rediscovered by M. Riesz.
It is also well known that there are Cauchy type reproducing formulas for solutions to such 
Dirac equations, which generalize classical function theory for analytic functions in the 
complex plane to some extent. These Cauchy integrals are fundamental to this paper.

The spin integral equation \eqref{eq:rotintro} solves the boundary value problem in the
unbounded exterior domain $\Omega$, for a prescribed tangential part of the fields at $\partial\Omega$,
in parallel with an auxiliary and complementary boundary value problem in the bounded interior
domain. We do not choose this auxiliary boundary value problem to be with prescribed normal part of the fields at the boundary,
which is the straightforward choice, since this would cause non-invertibility of the integral equation
at some discrete set of real resonances. Instead, the novelty with the spin integral equation
is that we impose a non-selfadjoint boundary condition for the auxiliary interior problem, which is a natural local elliptic boundary condition 
for the spin Dirac operator. This ensures that the spin integral equation is invertible for any
wave number $\im k\ge 0$, $k\ne 0$, and therefore solves the problem with spurious interior resonances 
in computational electromagnetics.

A problem with the combined field integral equation for Maxwell's equations, which is a classical way to modify the 
integral equations to avoid spurious interior resonances, see for example
Colton and Kress~\cite[Thm. 6.19]{CK},
is the low frequency breakdown, in that the numerical accuracy deteriorates as $k\to 0$.
There are integral equations available for Maxwell scattering which do not suffer from such low-frequency 
breakdown or interior spurious resonances: The generalized Debye potentials of Epstein and Greengard~\cite{EG1}.
Their method however requires the inversion of surface Laplace equations.
Our spin integral equation does not require any extra complications like that.
Furthermore the operator depends holomorphically
on $k$ and the only thing that happens at $k=0$ is that the operator fails to be invertible, 
but is still a Fredholm operator.

Unlike classical integral equations for Maxwell's equations, the spin integral equation uses function spaces without
any algebraic constraints, like tangential vector fields, 
or any differential constraints, like control of surface divergence.
In the combined field integral equation for Maxwell's equations as well as in 
the classical integral equation for solving the Neumann Helmholtz equation without spurious interior
resonances, there appears hypersingular integral 
operators.
In \eqref{eq:rotintro}, the integral operator is a singular integral, but it is known to be bounded for example  on $L_p$, $1<p<\infty$, and H\"older spaces 
$C^\alpha$, $0<\alpha<1$, on smooth surfaces.
Boundedness on $L_p$ in fact only requires a Lipschitz boundary.
However, the spin integral operator $h(x)\mapsto M(x)\pv\int_{\partial\Omega} \Psi_k(y-x) h(y) dy$ is never compact, 
being the product of an invertible singular integral operator and an invertible multiplication operator.
Nevertheless, the spin integral equation is bounded and invertible in $L_2(\partial\Omega)$ on any Lipschitz surface,
and for smooth surfaces also in $L_p(\Omega)$, $1<p<\infty$, and in H\"older space $C^\alpha(\partial\Omega)$, $0<\alpha<1$.

The outline of this paper is as follows.
The reader who wants a self contained formulation of the algorithm for solving the 
Maxwell and Helmholtz scattering problems with the spin integral equation, finds this in Section~\ref{sec:alg}.
The reader who want a derivation of the spin integral equation finds it in Sections~\ref{sec:bvpss}
and \ref{sec:constr}. Before this derivation, some background on Clifford algebra, Dirac equations
and Cauchy integrals is surveyed in Section~\ref{sec:Cauchyint}.
In the final Section~\ref{sec:compare} we explain the relationship between the spin integral 
equation and the classical double and single, and magnetic dipole integral equations for solving
the Helmholtz and Maxwell scattering problems. 

It is our hope that the spin integral equation formulated in this paper will prove useful for the
numerical solution of scattering problems, in particular for the Maxwell perfect conductor scattering problem.
In this paper, we have aimed to \Rd simplify \Bk the algebra by  limiting ourselves to the three dimensional Clifford algebra,
which only makes use of the classical scalar and vector products, in a clever combination to 
yield an associative Clifford product. 
But similar algorithms work in two dimensions as well as in higher dimension,
see \cite{Ax}. In even dimensions $2n$, it will use $2^n\times 2^n$ complex matrices (not pairs)
and a Hankel function of the first kind for the integral kernel.

\section{The spin integral equation}    \label{sec:alg}

In this section we state in detail our algorithm for solving the Maxwell and Helmholtz scattering problems
considered in this paper.
Throughout this paper we assume that $\R^3= \Omega^+\cap \partial \Omega\cup\Omega$ disjointly,
where $\Omega^+$ is a bounded domain with strongly Lipschitz regular boundary $\partial\Omega= \partial \Omega^+$.
That is we assume that $\partial\Omega$ locally is the graph (in some direction) of an at least Lipschitz 
regular function.
Denote by $n$ the unit vector field on $\partial\Omega$ pointing into $\Omega=\Omega^-$, which is at least a 
measureable vector field.
We further assume that the unbounded domain $\Omega$ is connected.
Throughtout this paper, we assume that the wave number $k$ appearing in Maxwell's and Helmholtz'
equations satisfies
$$
k\in \C\setminus\{0\} \text{ and } \im k\ge 0.
$$

\subsection{The scattering problems}

(M)
Consider Maxwell's equations in $\Omega$ consisting of a uniform, isotropic and conducting material, so that electric permittivity $\epsilon$,
magnetic permeability $\mu$ and conductivity $\sigma$ are constant and scalar.
We study electromagnetic wave propagation in $\Omega$, with material
constants $\epsilon, \mu,\sigma$, and $\Omega^+$ is assumed to be a perfect conductor.
Maxwell's equations in $\Omega$ then take the form 
\begin{equation}  \label{eq:maxwell}
  \begin{cases}
     \nabla \cdot  H=0, \\
     \nabla \times E= ik H, \\
     \nabla\times H=-ikE, \\
     \nabla\cdot E=0,
  \end{cases}
\end{equation}
for the (rescaled) electric and magnetic fields $E$ and $H$, where the wave number $k$ satisfies
$k^2= (\epsilon+ i\sigma/\omega)\mu \omega^2$.
The data in the scattering problem we seek to solve are incoming electric and magnetic fields $E^0$ and $H^0$
solving \eqref{eq:maxwell} in $\Omega$.
Our problem is to solve for the scattered electric and magnetic fields $E$ and $H$ solving \eqref{eq:maxwell} in $\Omega$, an inhomogeneous boundary condition at $\partial\Omega$ and the Silver--M\"uller radiation condition at infinity.
At $\partial\Omega$, since the fields vanish
in the perfect conductor $\Omega^+$, we have boundary conditions
\begin{equation}  
  \begin{cases}
     n \cdot  (H+H^0)=0, \\
     n \times (E+E^0)= 0, \\
     n\times (H+H^0)=J_s, \\
     n\cdot (E+E^0)=\rho_s.
  \end{cases}
\end{equation}
The radiation condition is that
\begin{equation}   \label{eq:silvermuller}
     \begin{cases}
     \tfrac x{|x|} \cdot  H(x)=o(|x|^{-1} e^{\im k |x|}), \\
     \tfrac x{|x|} \times E(x)- H(x)  =o(|x|^{-1} e^{\im k |x|}),\\
     \tfrac x{|x|} \times H(x)+E(x)=o(|x|^{-1} e^{\im k |x|}), \\
     \tfrac x{|x|} \cdot E(x)=o(|x|^{-1} e^{\im k |x|}),
  \end{cases}
\end{equation}
at infinity.
We remark that for $\im k>0$, the exponential bound on the growth combined with Maxwell equations will
in fact self-improve to an exponential decay.  

(H)
Turning to acoustic wave propagation, we also consider the
 Helmholtz equation
\begin{equation}    \label{eq:helmh}
   \Delta u +k^2 u=0
\end{equation}
in $\Omega$ with wave number $k$, for a scalar function $u:\Omega\to \C$.
We assume that the incoming wave $u^0$ solves \eqref{eq:helmh} in $\Omega$, and we want to 
compute a scattered wave $u$ also solving \eqref{eq:helmh} in $\Omega$, and satisfying either
the Dirichlet (sound soft) boundary conditions
$$
  u+u^0=0,
$$
or alternatively the Neumann (sound hard) boundary conditions
$\frac{\partial u}{\partial n} +  \frac{\partial u^0}{\partial n} =0$,
on $\partial\Omega$, and at infinity the Sommerfield radiation condition
\begin{equation}    \label{eq:sommerfield}
  \frac{\partial u}{\partial r}-iku= o(|x|^{-1}e^{\im k |x|}).
\end{equation}

\subsection{The algorithm}  

From the fundamental solution 
$
  \Phi_k(x)= -\frac{e^{ik|x|}}{4\pi|x|}
$
for the Helmholtz operator $\lap+k^2$, we compute its gradient vector field 
$$
  \nabla \Phi_k(x)= \left( -\frac x{|x|^2}+ik\frac x{|x|} \right)\Phi_k(x).
$$
We represent vectors by pairs of complex $2\times 2$ matrices, using the classical Pauli spin matrices,
and in $(\C^{2\times 2})^2$, the linear space of pairs of complex $2\times 2$ matrices, we add and multiply
component-wise so that
$(a_1,a_2)+ (b_1,b_2)= (a_1+b_1,a_2+b_2)$ and
$(a_1,a_2)(b_1,b_2)= (a_1b_1,a_2b_2)$.
Denote the identity element $I=\Big( \begin{bmatrix} 1 & 0 \\ 0 & 1 \end{bmatrix},
  \begin{bmatrix} 1 & 0 \\ 0 & 1 \end{bmatrix} \Big)$.

\begin{defn}   \label{eq:rhorepr}
   Let $\{e_1, e_2, e_3\}$ be the standard basis for $\R^3$ and define
$$
  \sigma_1=\begin{bmatrix} 0 & 1 \\ 1 & 0 \end{bmatrix}, \quad
  \sigma_2=\begin{bmatrix} 0 & -i \\ i & 0 \end{bmatrix}, \quad
  \sigma_3=\begin{bmatrix} 1 & 0 \\ 0 & -1 \end{bmatrix}.
$$
Define the complex linear map $\rho: \R^3\to (\C^{2\times 2})^2$ given by
$$
  \rho(e_1)= (\sigma_1,\sigma_1), \quad
    \rho(e_2)= (\sigma_2,\sigma_2), \quad
  \rho(e_3)= (\sigma_3,-\sigma_3).
$$
\end{defn}

For example, the unit normal vector $n=n_1e_1+n_2e_2+n_3e_3=n(y)$ at $y\in \partial \Omega$ 
is represented by the matrix pair 
$$
  \rho(n)= \Big( \begin{bmatrix} n_3 & n_1-in_2 \\ n_1+in_2 & -n_3 \end{bmatrix} ,
  \begin{bmatrix} -n_3 & n_1-in_2 \\ n_1+in_2 & n_3 \end{bmatrix} \Big).
$$

The integral equation for solving the scattering problems makes use of the following operators acting on 
functions  $h: \partial \Omega\to (\C^{2\times 2})^2$.

\begin{itemize}
\item
The principal value 
singular integral operator
$$
  R_k h(x)= 2\pv\int_{\partial \Omega} \Phi_k(x-y) \left( \rho\Big(\frac{x-y}{|x-y|^2}-ik\frac{x-y}{|x-y|} \Big)-ik I\right) h(y)  dy,
$$  
$x\in\partial\Omega$,
where $dy$ denotes surface measure on $\partial\Omega$.

\item
The non-singular integral operator corresponding to $R_k$ is
$$
  R^\Omega_k h(x)= -\int_{\partial \Omega} \Phi_k(x-y) \left( \rho\Big(\frac{x-y}{|x-y|^2}-ik\frac{x-y}{|x-y|} \Big)-ik I \right) h(y) dy,
$$
for $x\in\Omega$. Note our choice of different normalization.

\item
The multiplication operator 
$$
  (Mh)(x)= h(x)+ \inv{h(x)}\rho(n(x))+ \rho(n(x))\inv{h(x)}\rho(n(x)), \qquad x\in\partial\Omega.
$$
Here $\inv h$ denotes an auxiliary {\em involution} operation in $(\C^{2\times 2})^2$, given by
$$
  \inv h= \Big( \begin{bmatrix} b_{11} & -\rev a_{21} \\ -\rev a_{12} & b_{22} \end{bmatrix},
  \begin{bmatrix} a_{11} & -\rev b_{21} \\ -\rev b_{12} & a_{22} \end{bmatrix} \Big),
$$
  for
$h= \Big( \begin{bmatrix} a_{11} & a_{12} \\ a_{21} & a_{22} \end{bmatrix},
  \begin{bmatrix} b_{11} & b_{12} \\ b_{21} & b_{22} \end{bmatrix} \Big)$.

\end{itemize}

We can now state our algorithm for Maxwell and Helmholtz scattering.
  Define boundary data $g$ as follows.
  
(M)
  For Maxwell's equations,
  let $E^0$ and $H^0$ be incoming electric and magnetic fields.
   Define the tangential trace $E^0_T:= E^0|_{\partial \Omega}- (E^0|_{\partial \Omega}\cdot n)n= 
   (\epsilon_1, \epsilon_2, \epsilon_3)$ of the electric field and the
  normal trace $H^0_N:= (H^0|_{\partial \Omega}\cdot n)n= (\eta_1,\eta_2,\eta_3)$ of the magnetic field, which will satisfy the constraint
  $\nabla_T\times E^0_T = ik H^0_N$, and let
\begin{multline*}
  g=   - \Big( \begin{bmatrix} \epsilon_3 +i\eta_3 & \epsilon_1+\eta_2+i(-\epsilon_2+\eta_1) \\ \epsilon_1-\eta_2+i(\epsilon_2+\eta_1) & -\epsilon_3-i\eta_3 \end{bmatrix} , \\
  \begin{bmatrix} -\epsilon_3 +i\eta_3 & \epsilon_1-\eta_2+i(-\epsilon_2-\eta_1) \\ \epsilon_1+\eta_2+i(\epsilon_2-\eta_1) & \epsilon_3-i\eta_3  \end{bmatrix} \Big).
\end{multline*}

(H)
  For Helmholtz equations,
  let $u^0$ be an incoming wave satisfying Helmholtz equation \eqref{eq:helmh} in $\Omega\subset \R^3$.
  To solve the Dirichlet Helmholtz scattering problem, write the tangential gradient as
  $\nabla_T u^0=(f_1^\ta, f_2^\ta, f_3^\ta)$ at $\partial\Omega$, and
  let
$$
  g= 
 - \Big( \begin{bmatrix} iku^0+f_3^\ta  & f_1^\ta-if_2^\ta \\ f_1^\ta+if_2^\ta & iku^0-f_3^\ta  \end{bmatrix} , \\
   \begin{bmatrix} iku^0-f_3^\ta  & f_1^\ta-if_2^\ta \\ f_1^\ta+if_2^\ta & iku^0+f_3^\ta  \end{bmatrix}  \Big).
$$
 To solve the
 Neumann Helmholtz scattering problem instead, write the normal gradient as
  $\tfrac {\partial u^0}{\partial n}n=(f_1^\no, f_2^\no, f_3^\no)$, and
  let
$$
  g= 
 - \Big( \begin{bmatrix} if_3^\no & f_2^\no+if_1^\no \\ -f_2^\no+if_1^\no & -if_3^\no  \end{bmatrix} , \\
  \begin{bmatrix} if_3^\no  & -f_2^\no-if_1^\no \\ -f_2^\no-if_1^\no & -if_3^\no  \end{bmatrix}  \Big).
$$

\begin{alg}   \label{alg:rot}
Assume $\im k\ge 0$, $k\ne 0$.
Given such boundary data $g:\partial \Omega\to (\C^{2\times 2})^2$ as above, depending on which scattering
problem we consider, solve the singular integral equation
\begin{equation}   \label{eq:spinmatrixeq}
   h(x)- (MR_k h)(x)=2g(x)+2 \inv{g(x)}\rho(n(x)),\qquad x\in\partial\Omega,
\end{equation}
for $h:\partial \Omega\to (\C^{2\times 2})^2$.
Compute
$$
  F(x)= (R_k^\Omega h)(x), \qquad x\in\Omega,
$$
and write this function $F:\Omega\to (\C^{2\times 2})^2$ as
\begin{multline*}
  F= 
   \Big( \begin{bmatrix} \alpha+a_3 +i(b_3+\beta) & a_1+b_2+i(-a_2+b_1) \\ a_1-b_2+i(a_2+b_1) & \alpha-a_3+i(-b_3+\beta) \end{bmatrix} , \\
  \begin{bmatrix} \alpha-a_3 +i(b_3-\beta) & a_1-b_2+i(-a_2-b_1) \\ a_1+b_2+i(a_2-b_1) & \alpha+a_3+i(-b_3-\beta)  \end{bmatrix} \Big).
\end{multline*}

Depending on which scattering problem we consider, we have the following solution.

(M)
In case of Maxwell data $g$, we have $\alpha=\beta=0$, and $E_n=a_n$ and $H_n=b_n$, $n=1,2,3$, are the components of the uniquely determined scattered electric and 
magnetic fields, which satisfies the boundary conditions at $\partial\Omega$, the Maxwell equations in $\Omega$ and the Silver--M\"uller radiation condition at infinity.

(H)
In case of Helmholtz Dirichlet data $g$, we have $b_n=0$, $n=1,2,3$, $\beta=0$,
and $\alpha=iku$ and $a_n= \partial_n u$, $n=1,2,3$, where $u$ is the uniquely determined scattered wave, which satisfies the Dirichlet boundary conditions at $\partial\Omega$, the Helmholtz equation in $\Omega$ and the Sommerfield radiation condition at infinity.

In case of Helmholtz Neumann data $g$, we have $a_n=0$, $n=1,2,3$, $\alpha=0$,
and $\beta=ikv$ and $b_n= \partial_n v$, $n=1,2,3$, where $v$ is the uniquely determined scattered wave, which satisfies the Neumann boundary conditions at $\partial\Omega$, the Helmholtz equation in $\Omega$ and the Sommerfield radiation condition at infinity.
\end{alg}

The derivation of the spin integral equation \eqref{eq:spinmatrixeq} is found in Section~\ref{sec:NSspin}.
It follows from the results in Section~\ref{sec:bvpss}, that the spin integral equation is invertible 
on $L_2(\partial\Omega)$ for example, for all $\im k\ge 0$, $k\ne 0$.
In Section~\ref{sec:constr}, we show why the particular structure of the datum $g$ will produce 
Maxwell and Helmholtz fields respectively.

\section{The three dimensional Cauchy integral}    \label{sec:Cauchyint}

The algorithm for solving the Maxwell scattering problem, presented in Section~\ref{sec:alg},
was formulated in terms of matrices, as this is well suited for computatons.
For a geometrical understanding, we prefer instead the underlying Clifford algebra framework.
What we used in Section~\ref{sec:alg} was indeed the well known fact from representation theory that
the three dimensional complex Clifford algebra is isomorphic to the matrix-pair algebra $(\C^{2\times 2})^2$,
and $\rho$ from Definition~\ref{eq:rhorepr} is an example of such an isomorphism.

The three dimensional complex Clifford algebra, as a complex linear space, is the same as the
exterior algebra $\wedge \C^3$, and consists of objects, here referred to as multivectors, which we write
\begin{multline}   \label{eq:multvec3D}
  w= \alpha + a_1e_1+ a_2e_2+a_3e_3 \\
  + b_1e_2\wedg e_3+ b_2e_3\wedg e_1+ b_3e_1\wedg e_2+ \beta e_1\wedg e_2\wedg e_3,
\end{multline}
with complex coordinates $\alpha, a_n, b_n, \beta\in\C$.
Here $\{e_1,e_2,e_3\}$ is the standard basis for $\R^3\subset \C^3$, which induces the basis
$\{1, e_1, e_2, e_3, e_2\wedg e_3, e_3\wedg e_1, e_1 \wedg e_2, e_1\wedg e_2\wedg e_3\}$ for the eight-dimensional 
complex linear space $\wedge \C^3$. 
We equip $\wedge\C^3$ with the complex inner product for \Rd which \Bk the above basis is ON.
Note that we consider not only the vectors $\C^3$ but also the scalars $\C$ as subspaces in $\wedge \C^3$, by viewing $1$ as a basis element in $\wedge \C^3$.
Unlike in higher dimension, the geometry of multivectors in $\wedge \C^3$ is not much beyond 
vectors, since the Hodge star map
\begin{multline*}
  {*w}= \alpha e_1\wedg e_2\wedg e_3 + a_1e_2\wedg e_3+ a_2e_3\wedg e_1+a_3e_1\wedg e_2\\
  + b_1e_1+ b_2e_2+ b_3e_3+ \beta,
\end{multline*}
can be used to identify the one-dimensional vectors $\{e_1,e_2,e_3\}$ and the two-dimensional bivectors
$\{e_2\wedg e_3, e_3\wedg e_1, e_1\wedg e_2\}$, as well as
the zero-dimensional scalars $\{1\}$ and the $3$-vectors $\{e_1\wedg e_2\wedg e_3\}$. 

On the complex linear space $\wedge \C^3$, there are two natural associative, but non-commutative, products.
\begin{itemize}
\item The exterior product $\wedg$, relates to the affine structure of $\R^3$ and satisfies $u\wedg u=0$ for all vectors $u\in \C^3$. In terms of scalar and
vector products, the exterior product of a vector $u$ and a multivector $w= \alpha+ a+*b+*\beta$ is
\begin{equation}  \label{eq:extprod}
  u\wedg w= 0 + u\alpha+ *(u\times a)+ *(u\cdot b).
\end{equation}

\item The Clifford product, denoted by juxtaposition, encodes also the euclidean geometry and satisfies $uu= |u|^2$ for all vectors $u\in\C^3$. Note the sign convention we chose here.
In terms of scalar and vector products, the Clifford product of a vector $u$ and a multivector $w= \alpha+ a+*b+*\beta$ is
\begin{equation}   \label{eq:Cliffprod}
  u w= u\cdot a+(u\alpha-u\times b)+ *(u\times a+u\beta)+ *(u\cdot b).
\end{equation}
\end{itemize}

Unlike the exterior product, the Clifford product yields an algebra which is essentially isomorphic 
to a matrix algebra. In the above complex three dimensional case, the following is an explicit but arbitrary 
choice of isomorphism which we use in this paper.

\begin{lem}   \label{lem:matrixrepr}
  The space of multivectors $\wedge \C^3$, equipped with the Clifford product, is isomorphic to the 
  space $(\C^{2\times 2})^2$ of pairs of complex $2\times 2$ matrices, equipped with pairwise matrix
  multiplication as in Section~\ref{sec:alg}, via the isomorphism
\begin{multline*}
  \rho(w)= 
   \Big( \begin{bmatrix} \alpha+a_3 +i(b_3+\beta) & a_1+b_2+i(-a_2+b_1) \\ a_1-b_2+i(a_2+b_1) & \alpha-a_3+i(-b_3+\beta) \end{bmatrix} , \\
  \begin{bmatrix} \alpha-a_3 +i(b_3-\beta) & a_1-b_2+i(-a_2-b_1) \\ a_1+b_2+i(a_2-b_1) & \alpha+a_3+i(-b_3-\beta)  \end{bmatrix} \Big),
\end{multline*}
for $w\in\wedge \C^3$ as in \eqref{eq:multvec3D}.
\end{lem} 

An algebraic operation that we also need, is the involution $\inv w$ of a multivector $w$, 
which is 
\begin{multline}   \label{eq:involution}
  \inv w= \alpha - a_1e_1- a_2e_2-a_3e_3 \\
  + b_1e_2\wedg e_3+ b_2e_3\wedg e_1+ b_3e_1\wedg e_2- \beta e_1\wedg e_2\wedg e_3,
\end{multline}
that is negation of vectors, extended to multivectors.
Note that under the representation $\rho$, this involution corresponds to $\hat h$ from 
Definition~\ref{eq:rhorepr}.
It is an automorphism of $\wedge \C^3$, both with respect to the exterior and Clifford products,
which is very useful. One first example is a Riesz formula
\begin{equation}   \label{eq:riesz}
  u\wedg w= \tfrac 12(uw+\inv w u),
\end{equation}
with expresses the exterior product by a vector $u$, in terms of the Clifford product.

\begin{defn}
The (Hodge--) Dirac is the formally elliptic first order partial differential operator $\dirac$ \Rd given by \Bk
$$
\dirac F= e_1(\pd_1 F)+ e_2(\pd_2 F)+ e_3(\pd_3 F),
$$
that is the nabla symbol acting via the Clifford product.
\end{defn}

\begin{ex}
The relevance of Clifford algebra to this paper is as follows.

(M)
In the $\wedge \C^3$ framework, the magnetic field is a bivector field $*H$,  
and we combine the electric and magnetic fields in one multivector field
\begin{equation}  \label{eq:EMfield3D}
  F= E+*H=E_1 e_1+E_2e_2+E_3e_3+ H_1e_2\wedg e_3+ H_2e_3\wedg e_1+H_3e_1\wedg e_2,
\end{equation} 
which has zero scalar and $3$-vector parts.
In this way, Maxwell's equations \eqref{eq:maxwell} can be written in 
a compact way as
\begin{equation}   \label{eq:diraceq}
  \dirac F= ikF.
\end{equation}
Indeed, by \eqref{eq:Cliffprod} we see that
\begin{multline*}
  \dirac F= \nabla (E+*H)= \nabla\cdot E+(0-\nabla\times H) \\+ *(\nabla\times E+0)+ *(\nabla\cdot H)
  = ik(0+E+*H+*0).
\end{multline*}

(H)
In the $\wedge \C^3$ framework, we combine the acoustic wave function $u$
and its gradient, into one multivector field
\begin{equation*}  
  F= iku+\nabla u=iku+(\partial_1)u e_1+(\partial_2u)e_2+(\partial_3u)e_3,
\end{equation*} 
which has zero bivector and $3$-vector parts.
In this way, Helmholtz' equation for $u$ is equivalent to \eqref{eq:diraceq}
for $F$.
Indeed, by \eqref{eq:Cliffprod} we see that
$$
  \dirac F= \nabla (iku+\nabla u)= \nabla\cdot \nabla u+ik\nabla u  + *0+ *0
  = ik(iku+\nabla u+*0+*0).
$$
For the Neumann problem, we shall instead use that the Helmholtz equation for $u$
is equivalent to \eqref{eq:diraceq} for $F=*\nabla u+ ik{*u}$.
\end{ex}

The fundamental solution $1/z$ to the Cauchy--Riemann $\conj \partial$ operator, yields 
the Cauchy reproducing formula for analytic functions in the complex plane.
Generalizing this to three dimensions and $k\ne 0$, one obtains a Cauchy type 
reproducing formula for solutions to $\dirac F= ikF$.
The Dirac operator is the famous differential square root of the Laplace operator: 
$\dirac^2 F= \lap F$ acting component-wise on the multivector field, and therefore
$$
  (\dirac +ik)((\dirac-ik)\Phi_k)(x)= (\lap+k^2)\Phi_k(x)= \delta_0(x),
$$
from which we obtain a fundamental solution 
$$
  \Psi_k(x)= (\dirac-ik)\Phi_k(x)= \Big( - \frac{x}{|x|^2}+ik\big( \tfrac x{|x|}-1\big)\Big)\Phi_k(x)
$$ 
to $\dirac +ik$.
Stokes' theorem shows that
\begin{multline*}
  \int_{\partial\Omega} \Psi_k(y-x) (-n(y))F(y) dy \\
  = \int_\Omega \big( (\dirac_y\Psi_k(y-x)) F(y)+ \Psi_k(y-x) (\dirac F(y)) \big) dy \\
  = F(x)+ \int_\Omega \big( -ik\Psi_k(y-x) F(y)+ \Psi_k(y-x) ik F(y) \big) dy 
  =F(x),
\end{multline*}
when $x\in\Omega$, if $\dirac F= ikF$ in $\Omega$ and if  the Dirac radiation condition
 \begin{equation}   \label{eq:diracrad}
    (\tfrac x{|x|}-1)F(x)= o(|x|^{-1}e^{\im k |x|}),\qquad\text{as }  x\to \infty,
 \end{equation}
  holds. For the electromagnetic field \eqref{eq:EMfield3D} this is exactly the Silver--M\"uller radiation condition,
as seen from \eqref{eq:Cliffprod}.

\begin{lem}   \label{lem:radiationsuffice}
  Assume that $\dirac F=ikF$ in $\Omega$ and that $F$ satisfies the 
  Dirac radiation condition \eqref{eq:diracrad}.
  Then
  $$
  \lim_{R\to\infty}\int_{|x|=R} \Psi_k(y-x) \tfrac {y}{|y|}F(y) dy=0.
  $$
\end{lem}

\begin{proof}
   Note that left multiplication by $\tfrac 12(1-\tfrac x{|x|})$ is a projection in $\wedge\C^3$,
   so the condition is on one half of the multivector field $F$. 
      To obtain some control of the other half of $F$, we integrate the identity 
   $$
     2|F|^2= |(\tfrac x{|x|}-1)F|^2+ 2(\tfrac x{|x|}F,F),
   $$
   over the a large sphere $|x|=R$ and use Stokes' theorem on the last term to obtain
   \begin{multline*}
     2\int_{|x|=R} |F|^2 dx= \int_{|x|=R} |(\tfrac x{|x|}-1)F|^2 dx +2 \int_{|x|=R_0} (\tfrac x{|x|}F,F) dx \\
     -4\im k\int_{R_0<|x|<R} |F|^2 dx\le Ce^{2\im k R },
   \end{multline*}
   since $(\dirac F,F)+ (F,\dirac F)= -2\im k|F|^2$ and $\im k\ge 0$.
   The stated limit now follows by using this estimate for the term $-x/|x|^2$, and the hypothesis for the
   term $ik(x/|x|-1)$, in $\Psi_k$.
\end{proof}

We define the trace
of Cauchy extensions
$$
   C^-_k h(x)=\lim_{z\in\Omega, z\to x} C^\Omega_k h(z),  \qquad x\in \partial\Omega,
$$
where $C^\Omega_k h(z)= \int_{\partial\Omega} \Psi_k(y-z) (-n(y))h(y) dy$, $z\in\Omega$, and
$h:\partial\Omega\to \wedge\C^3$.
Similarly for the complementary bounded domain $\Omega^+$, we set 
$$
   C^+_k h(x)=\lim_{z\in\Omega^+, z\to x} C^{\Omega^+}_k h(z),  \qquad x\in \partial\Omega,
$$
where $C^{\Omega^+}_k h(z)= \int_{\partial\Omega} \Psi_k(y-z) n(y)h(y) dy$, $z\in\Omega^+$. 
The two operators $C^\pm_k$ act on functions $h:\partial \Omega\to \wedge\C^3$, and we observe that
$$
  (C^-_k)^2= C^-_k, \qquad (C^+_k)^2= C^+_k,\qquad C^+_kC^-_k=0= C^-_kC^+_k,
$$
as a consequence of Stokes theorem as above.
Moreover, we compute the limits 
$$
  C^\pm_k h= \tfrac 12(h \pm C_k h),
$$
where $C_k$ denotes the principal value singular integral
$$
   C_k h(x)=2 \pv \int_{\partial\Omega} \Psi_k(y-x) n(y)h(y) dy, \qquad x\in\partial\Omega.
$$
In particular, $C^+_k+C^-_k= I$ and we conclude that $C^\pm_k$ are two complementary 
projections, and that $C_k= C_k^+-C_k^-$ is a reflection operator, meaning that $C_k^2= I$.

\section{Boundary conditions and well posedness}   \label{sec:bvpss}

By $\mH$ we denote a suitable Banach space of functions $h:\partial\Omega\to \wedge\C^3=(\C^{2\times 2})^2$.
Throughout the remainder of this paper we shall use
$$
  \mH= L_2(\partial\Omega, \wedge\C^3).
$$
However, for smooth surfaces we could equally well use an $L_p$, $1<p<\infty$, or H\"older space $C^{\alpha}$, $0<\alpha<1$, for example.
What is needed is that the Cauchy singular integral operator $C_k$, as well as the multiplication operators being used, act boundedly in $\mH$, and that Fredholm well posedness holds for the $N$ and $S$ boundary value problems, as discussed below. It is for the latter  reason that we prefer $L_2$, since well posedness holds here for all strongly Lipschitz surfaces.

As we saw in Section~\ref{sec:Cauchyint}, the Dirac equation $\dirac F= ik$ in the two domains $\Omega^+$ and $\Omega= \Omega^-$ induce a topological (but non-orthogonal) splitting of the function space $\mH$ on $\partial\Omega$ into closed
 Hardy subspaces
 $$
    \mH= C_k^+\mH\oplus C_k^-\mH.
 $$ 
 Below we make use of the formalism from \cite{Ax1} for boundary value problems, which means
 that we encode the boundary conditions in a second reflection operator $A$.
 Well posedness of all the associated boundary value problems means that $\pm 1$ do not belong 
 to the spectrum of the cosine operator 
 $$
 \tfrac 12(C_k A+AC_k)= A^+C_kA^+-A^-C_kA^-.
 $$
 An optimal boundary value problem is when $C_kA+AC_k=0$, or Fredholm optimal when 
 this cosine operator is compact.

\subsection{The $N$ boundary conditions}    \label{sec:Nbvp}

We first survey known estimates for normal and tangential boundary conditions.
See \cite{Ax2} for more details. In this case we choose $A$ to be the operator
$$
  (Nf)(x)= n(x)\inv{f(x)}n(x), \qquad x\in \partial\Omega.
$$
The associated projection operators $N^\pm$ are 
$$
N^+f= \tfrac 12(f+Nf)= n(n\wedg f),
$$
which
gives the tangential (to the tangent plane $T_y\partial \Omega$) part of $f$, and $N^-f= \tfrac 12(f-Nf)$, which gives the normal part of $f$. 

When $\partial \Omega$ is smooth, we see that the cosine operator $ \tfrac 12(C_k N+NC_k)$ is compact.
Indeed the kernel of this integral operator is only weakly singular, since $\Psi_k(x-y)$ is almost
tangential, and therefore nearly anticommutes with $n(x)\approx n(y)$, when $x\approx y$.
In the general case when $\partial\Omega$ is merely Lipschitz regular, the cosine operator
fails to be compact, but it is still possible to prove that its essential spectrum does not contain $\pm 1$, by using 
Rellich type identities.
In our setting of the Dirac equation, one considers the sesquilinear form
$$
  \mH\times\mH\to\C: f,g\mapsto \Theta(f,g)= \int_{\partial\Omega}(n(x)\inv{f(x)}, g(x)\theta(x)) dx,
$$
where $\theta\in C^\infty_0(\R^3;\R^3)$ is a smooth vector field chosen so that
$(\theta(x),n(x))\ge C>0$ on $\partial\Omega$. 
On the one hand, by Stokes theorem, $\Theta(f,f)$ is essentially zero when $f$ is in one of the Hardy subspaces
$C_k^+\mH$ and $C_k^-\mH$. On the other hand, $\Theta$ is positive definite on $N^+\mH$ and
negative definite on $N^-\mH$. 
As in \cite[Thm. 3.4]{Ax2}, specializing to $V=N$, one deduce from this that $\lambda I-\tfrac 12(C_kN+ NC_k)$ is a Fredholm
operator with index zero for all $\lambda$ in a hyperbolic neighbourhood of $[1,\infty)$ and $(-\infty,-1]$.

The two reflection operators $C_k$ and $N$ describe four boundary value problems,
 where we seek solutions to $\dirac F=ikF$ in $\Omega^+$ or $\Omega^-$, with prescribed 
 tangential part $N^+(F|_{\partial\Omega})$ or normal part $N^-(F|_{\partial\Omega})$ of the multivector
 field at the boundary.
 Well posedness of such a problem is equivalent to the invertibility of a restricted projection,
as follows.
 
 \begin{thm}   \label{thm:Nwp}
   The four restricted projections
$$
  N^+: C_k^-\mH\to N^+\mH,
$$
$N^-: C_k^-\mH\to N^-\mH$, $N^+: C_k^+\mH\to N^+\mH$
and $N^-: C_k^+\mH\to N^-\mH$
are Fredholm operators with index zero for any $k\in\C$.
They fail to be invertible only on a discrete set of resonances in $\im k\le 0$.
However, the two restricted projections $N^\pm: C_k^-\mH\to N^\pm\mH$ corresponding to the exterior domain $\Omega$, when this is connected, will also be invertible for all $k\in\R\setminus\{0\}$.
\end{thm}
 
\begin{proof}
  All the statements, except the invertibility of the exterior restricted projections for $k\in \R\setminus\{0\}$,
  follow from the Rellich identities, Fredholm perturbation theory and analytic Fredholm theory.
  See \cite[Thm. 3.5]{Ax2}, and specialize to $V=N$.

  To prove the last statement, assume that $F$ satisfies  the Dirac equation $\dirac F= ikF$ in $\Omega$, the radiation 
  condition   $(\tfrac x{|x|}-1)F(x)= o(|x|^{-1})$ as $x\to \infty$, and 
  that either $N^+ F|_{\partial\Omega}= 0$ or $N^- F|_{\partial\Omega}= 0$ on $\partial \Omega$.
  It suffices to prove that $F=0$ identically in $\Omega$, provided $k\in\R\setminus\{0\}$ and $\Omega$ is connected.
  Note that each component function of $F$ will satisfy the scalar Helmholtz equation since
   $$
      (\lap+ k^2)F= (\dirac+ik)(\dirac -ik)F= (\dirac+ik)0=0.
   $$
   In particular it is a real analytic function in $\Omega$, and by Rellich's lemma,
   see for example~\cite[Ch. 9, Lem. 1.2]{Tay2},
   it suffices to show $\lim_{R\to\infty}\int_{|x|=R}|F|^2 dx =0$.
   To this end, calculate as in the proof of Lemma~\ref{lem:radiationsuffice}, but use Stokes' theorem
   on all $\Omega\cap \{ |x|<R\}$, to get
   $$
          2\int_{|x|=R} |F|^2 dx= \int_{|x|=R} |(\tfrac x{|x|}-1)F|^2 dx +2 \int_{\partial\Omega} (nF,F) dx.
   $$
   By hypothesis, the first term tends to zero, and the last term vanishes because of the assumed 
   boundary conditions.
\end{proof}

\subsection{The $S$ boundary conditions}

We now study two auxiliary boundary conditions described by the reflection operator
$$
  (Sf)(x)= n(x)f(x), \qquad x\in\partial\Omega.
$$
The associated projection operators $S^\pm$ are 
$$
S^-f= \tfrac 12(f-Sf)= \tfrac 12(1-n)f,
$$
and $S^+f= \tfrac 12(f+Sf)= \tfrac 12(1+n)f$.
In this case 
$$
C_k Sf(x)= 2 \pv \int_{\partial\Omega} \Psi_k(y-x) f(y) dy, \qquad x\in\partial\Omega,
$$
is a compact perturbation of a skew-adjoint operator, 
and therefore so is the cosine operator $\tfrac 12(C_kS+SC_k)$.
When $\partial\Omega$ is smooth, one moreover sees from the almost orthogonality of $\Psi_k(x-y)$ and 
$n(x)\approx n(y)$, when $x\approx y$, that the cosine operator is a weakly singular, and hence compact, operator.
Thus, for any Lipschitz surface, $\pm 1$ will not be in the essential spectrum of this cosine operator.

The two reflection operators $C_k$ and $S$ describe four boundary value problems,
 where we seek solutions to $\dirac F=ikF$ in $\Omega^+$ or $\Omega^-$, with prescribed 
  $S^+(F|_{\partial\Omega})$ part or $S^-(F|_{\partial\Omega})$ part of the multivector
 field at the boundary.
 Well posedness of such a problem is equivalent to the invertibility of a restricted projection,
as follows.

 \begin{thm}   \label{thm:Swp}
   The four restricted projections
$$
  S^-: C_k^+\mH\to S^-\mH,
$$
$S^+: C_k^-\mH\to S^+\mH$, $S^+: C_k^+\mH\to S^+\mH$
and $S^-: C_k^-\mH\to S^-\mH$
are Fredholm operators with index zero for any $k\in\C$.
They fail to be invertible only on a discrete set of resonances $k$.
The restricted projection $S^-: C_k^+\mH\to S^-\mH$ is invertible when $\im k\ge 0$.
\end{thm}
 
\begin{proof}
  All the statements, except the invertibility results, follow from the skew-adjointness, modulo compact
  operators,
  Fredholm perturbation theory and analytic Fredholm theory, similar to
  the proof of \cite[Thm. 3.5]{Ax2}, replacing $V$ there by $S$.

   To prove injectivity for $S^-: C_k^+\mH\to S^-\mH$, we apply Stokes' theorem to get
$$
  \int_\Omega \Big( (\dirac F, F)+ (F,\dirac F)\Big) dx= \int_{\partial \Omega} (nF,F) dy.
$$
If $\dirac F= ikF$ and $F|_{\partial\Omega}\in S^+\mH$, and hence $nF=F$, then
we obtain
$-2\im k\|F\|^2_{L_2(\Omega)}= \|F\|^2_{L_2(\partial\Omega)}$.
If $\im k>0$, clearly $F=0$ identically in $\Omega$.
If $\im k=0$, then $F|_{\partial\Omega}=0$. Using the Cauchy reproducing formula, it follows
that $F=0$ in all $\Omega$.
This proves the injectivity of the restricted projection.
\end{proof}

\subsection{Mixed $N$ and $S$ boundary conditions}   \label{sec:NSspin}

We consider the following Dirac scattering problem.
Assume given a tangential multivector field
$$
   g= \alpha+ a+*b+*\beta\in N^+\mH
$$
on $\partial\Omega$. This means that $\alpha$ is arbitrary, $a$ is a tangential
vector field, $b$ is a normal vector field, and $\beta=0$.
We seek a multivector field $F:\Omega\to \wedge\C^3$ solving $\dirac F=ikF$ in $\Omega$,
with boundary condition
$$
  N^+(F|_{\partial\Omega})=g,
$$
and satisfying the radiation condition
$(\tfrac x{|x|}-1)F(x)= o(|x|^{-1}e^{\im k |x|})$
as $x\to \infty$.
Clearly, well posedness of this boundary value problem means exactly
that the restricted projection
$$
  N^+: C_k^-\mH\to N^+\mH
$$
is invertible.
We know that this is true when $\im k\ge 0$ and $k\ne 0$ by Theorem~\ref{thm:Nwp}.
The main idea in this paper is to combine this restricted projection with the auxiliary 
restricted projection $S^-: C_k^+\mH\to S^-\mH$, which is invertible for $\im k\ge 0$ by Theorem~\ref{thm:Swp}, 
to obtain an equivalent 
invertible integral equation on the full space $\mH$, in the following way.

\begin{thm}  \label{thm:Dbvpasspin}
  Let $g\in N^+\mH$.
  Then $f\in C_k^-\mH$ and $N^+ f=g$ if and only if 
  $f\in \mH$ solves the equation
$$
  \Big(I- (N+S+SN)C_k\Big) f= 2(I+S)g.
$$
\end{thm}

\begin{proof}
  We show that both statements are equivalent to
\begin{equation*} 
  (N^+C_k^- -S^-C_k^+)f=g.
\end{equation*}
Note that $N^+\mH\cap S^-\mH=\{0\}$, so this equation is equivalent to
$N^+C_k^-f=g$ and $S^-C_k^+f=0$.
By Theorem~\ref{thm:Swp}, the latter is equivalent to $f\in C_k^-\mH$.

Calculating, we get
\begin{multline*}
 N^+C_k^- -S^-C_k^+=
  \tfrac 12(N^+(I-C_k)-S^-(I+C_k)) \\
  =\tfrac 12((N^+-S^-)-(N^++S^-)C_k) =
  \tfrac 14((N+S)-(2I+N-S)C_k).
\end{multline*}
Multiplying the equation by $N+S$, using that $N$ and $S$ are anti-commuting reflection
operators, and $Ng=g$, completes the proof.
\end{proof}

We note that the spin integral equation in Algorithm~\ref{alg:rot} \Rd coincides \Bk with the 
equation in Theorem~\ref{thm:Dbvpasspin}, after multiplication of the equation by $S$,
under the relations
$(\C^{2\times 2})^2= \wedge \C^3$, 
$h= Sf$,
$R_k= C_kS$ and
$M= I +SN+ N$.

To obtain further understanding of the spin equation from Theorem~\ref{thm:Dbvpasspin}, we
may consider the operator algebra generated by the two reflection operators $N$ and $S$,
as in \cite{Ax1}.
As used above, the cosine operator here is $\tfrac 12(NS+SN)=0$.
In particular, we have a splitting
$$
  \mH= N^-\mH\oplus S^+\mH.
$$
One calculates that the reflection operator $A$ across $S^+\mH$, along
$N^-\mH$, is the operator
$$
 A= N+S+SN,
$$
appearing in the spin equation.
Thus the associated projections are $A^\pm= \tfrac 12(I\pm (N+S+SN))$, so that
$A^+g= g+Sg$ as in the right hand side of the spin equation.
We leave these identities to the reader to verify, as we shall not use them.

\section{Maxwell and Helmholtz constraints}   \label{sec:constr}   
  
We have seen in Section~\ref{sec:NSspin} how the spin integral equation solves the Dirac scattering problem into which
the Maxwell and Helmholtz scattering problems embed. 
Note from \eqref{eq:Cliffprod}, that the Silver--M\"uller radiation condition \eqref{eq:silvermuller}
is equivalent to the Dirac radiation condition \eqref{eq:diracrad} for $F=E+*H$,
where the first and last equations in \eqref{eq:silvermuller} are redundant.
It is also clear  from \eqref{eq:Cliffprod} that the Dirac radiation condition for $F= iku+\nabla u$ implies the 
Sommerfield radiation condition \eqref{eq:sommerfield} for $u$.
Conversely, the Sommerfield radiation condition implies a Green representation.
   See \cite[Thm. 2.4]{CK}, which can be extended to $\im k\ge 0$, $k\ne 0$.
 From this the control of the angular derivative 
  $$
    \tfrac x{|x|}\times \nabla u=  o(|x|^{-1}e^{\im k |x|})
  $$
follows, and as a consequence, $F$ will satisfies the Dirac radiation condition.

It now remains to see how the constraints on the boundary
data $g$ force the Dirac solution to be a Maxwell or Helmholtz field respectively.
For this, we make use of the following operators.

\begin{itemize}
\item
The {\em exterior derivative} is the first order partial differential operator
$$
d F= e_1\wedg(\pd_1 F)+ e_2\wedg(\pd_2 F)+ e_3\wedg(\pd_3 F),
$$
that is the nabla symbol acting via the exterior product.
\item
The {\em time reflection operator} $T$ is the involution
$$
  Tf= \inv f,
$$
with associated projections $T^+ f= \tfrac 12(f+\inv f)$ onto the spacelike part
of $f$, and $T^-f= \tfrac 12(f-\inv f) $ onto the timelike part of $f$.
\end{itemize}

(M)
We first consider the Maxwell constraint, and complete the derivation of Algorithm~\ref{alg:rot}
for Maxwell scattering.
 Lemmas~\ref{lem:constraint} and \ref{lem:constraint3dhelm} below may appear somewhat cryptical.
 However, they become more transparent when 
 using a spacetime formulation of the Dirac equation, as in 
 \cite{McMi, AGHMc, Ax}.

\begin{lem}   \label{lem:constraint}
  Assume that $F= \alpha+a+*b+*\beta$ satisfies
  the Dirac equation $\dirac F=ikF$.
  Then $\alpha=\beta=0$ if and only if
  $$
     dF= ikT^+F.
  $$
\end{lem}

\begin{proof}
  By \eqref{eq:Cliffprod}, $\dirac F= ikF$ means
 $$
    \nabla\cdot a+ (\nabla\alpha- \nabla\times b)+ *(\nabla\times a+\nabla \beta)+*(\nabla\cdot b)= 
    ik(\alpha+a+*b +*\beta).
 $$
  From \eqref{eq:extprod} we see that $dF= ikT^+F$ means
 $$
    0 + \nabla \alpha+ *(\nabla\times a)+ *(\nabla \cdot b)= ik(\alpha+0+*b+*0).
 $$
 From these formulas, the claim is straightforward to verify.
\end{proof}

  Consider the incoming electromagnetic field $F^0= E^0+ *H^0$, and the tangential
  part $g$ of its trace. We solve the spin integral equation \eqref{eq:spinmatrixeq} for $h$, and obtain a solution
  $F$ to the Dirac scattering problem.
  Now define the auxiliary multivector field 
  $$
    F' = T (d-ikT^+)F. 
  $$
  By Lemma~\ref{lem:constraint}, it suffices to show $F'=0$ in $\Omega$.
  Given the unique solvability of the Dirac scattering problem, it suffices to show
  $$
  \begin{cases}
     \dirac F'=ikF', & \text{in }\Omega, \\
      N^+(F'|_{\partial\Omega})=0, & \text{on }\partial\Omega, \\
      (x/|x|-1)F'= o(|x|^{-1}e^{\im k |x|}), & \text{at } \infty.
   \end{cases}
  $$
  To verify the Dirac equation, we calculate
  \begin{multline*}
    \dirac F'= -T(\dirac dF-ik\dirac T^+ F)= -T((\dirac -d)\dirac F-ikT^- \dirac F) \\
    = -ikT((\dirac F-ikF)-(dF- ik T^+F))= 0+ikF',
  \end{multline*}
  since $\dirac d+d\dirac = \dirac^2$ and $T^++T^-=I$.
  
  To verify the boundary condition we note that $N^+(F'|_{\partial\Omega}) = \imath^* (F')$, 
  viewing $F'$ as a direct sum of differential forms, where $\imath^*$ denotes the pullback 
  by the inclusion map $\imath:\partial\Omega\to \R^3$.
  By the well known commutation of pullbacks and exterior derivatives, we obtain
  $$
    N^+(F'|_{\partial\Omega})= \imath^*(T (d-ikT^+)F)= T(d_T-ik T^+)(\imath^* F)=0,
  $$
  since $\imath^* F= E_T + *H_N= -(E^0_T + *H^0_N)$, and
  $(d_T-ik T^+)(E^0_T+*H^0_N)= d_T E^0_T-ik {*H_N^0}= *(\nabla_T\times E^0_T-ikH^0_N)=0$.
  Here $d_T$ denotes the intrinsic exterior derivative on $\partial \Omega$,
  and we have used that $d_T(*H^0_N)=0$ since $\dim\partial\Omega=2$.
  
  To verify the Dirac radiation condition, we use the fundamental solution $\Psi_k$ and the 
  radiation condition for $F$ to show that it has a representation 
  $$
  F(x)= \int_{\R^3} \Psi_k(y-x)w(y) dy,
  $$
  for some $w\in C_0^\infty(\R^3;\wedge\C^3)$.
  Indeed, if $\eta\in C_0^\infty(\R^3;\R)$ \Rd satisfies \Bk $\eta=1$ on a neighbourhood of $\clos{\Omega^+}$,
  Stokes' theorem and the radiation condition shows
  $$
    0=\int  \dirac\big(\Psi_k(y-x)  (1-\eta)F\big)dy=F(x)+\int\Psi_k(y-x)\big((\dirac -ik)((1-\eta)F\big))dy, 
  $$
  so we can choose $w= -(\dirac -ik)((1-\eta)F)$.
  
  Then we calculate modulo $o(|x|^{-1}e^{\im k |x|})$ to obtain
  $$
    F'(x)\approx ik T\int_{\R^3} (d_x-ikT^+)\Big((\tfrac{y-x}{|y-x|}-1) \Phi_k(y-x) w(y) \Big) dy,
  $$
  and further let $z=(x-y)/|x-y|\approx x/|x|$ and calculate  
 \begin{multline*} 
   (d_x-ikT^+)\big((z+1) \Phi_k(x-y) w \big) \\
   \approx ik z\wedg((z+1)\Phi_k w)-ik( z(I-T^+)w+ T^+w) \Phi_k \\
  = ik \Big( \tfrac 12( z(z+1)w+(-z+1)\inv w z )  -zw+(z-1)(T^+w)\Big) \Phi_k \\
  = ik \Big( \tfrac 12( (z-1)zw-(z-1)\inv w z )  +(z-1)(T^+w)\Big) \Phi_k,
   \end{multline*}
   using \eqref{eq:riesz}.
  Since $(z-1)T= -T(z+1)$ and $(z+1)(z-1)=0$, it follows that $F'$ satisfies the radiation condition.
  This completes the derivation of Algorithm~\ref{alg:rot} for Maxwell's equations.

(H)
We next consider the Helmholtz constraint, and complete the derivation of Algorithm~\ref{alg:rot}
for Helmholtz scattering.

\begin{lem}   \label{lem:constraint3dhelm}
  Assume that $F= \alpha+a+*b+*\beta$ satisfies
  the Dirac equation $\dirac F=ikF$.
  Then $\nabla \alpha= ika$ and $\nabla \beta= ikb$, if and only if
  $$
     dF= ikT^-F.
  $$
  In this case $F_1= \alpha+ a$ solves $\dirac F_1=ikF_1$ and 
  $F_2=  *b+ *\beta$ solves $\dirac F_2=ikF_2$.
\end{lem}

\begin{proof}
  By \eqref{eq:Cliffprod}, $\dirac F= ikF$ means
 $$
    \nabla\cdot a+ (\nabla\alpha- \nabla\times b)+ *(\nabla\times a+\nabla \beta)+*(\nabla\cdot b)= 
    ik(\alpha+a+*b +*\beta).
 $$
  From \eqref{eq:extprod} we see that $dF= ikT^-F$ means
 $$
    0 + \nabla \alpha+ *(\nabla\times a)+ *(\nabla \cdot b)= ik(0+ a+0+*\beta).
 $$
 From these formulas, the claim is straightforward to verify.
\end{proof}

Given Dirichlet and Neumann data $g_D= -(iku^0+\nabla_T u^0)$ and $g_N= -{*(\tfrac{\partial u^0}{\partial n}n)}$, we solve the spin integral equation \eqref{eq:spinmatrixeq} for $h$,
with boundary data $g_D+g_N$, and obtain a solution
  $F$ to the Dirac scattering problem.
We may of course set $g_N=0$ and solve only the Dirichlet problem, or vice versa.

To show that the Dirichlet and Neumann parts of $F$ are decoupled as in Lemma~\ref{lem:constraint3dhelm}, we consider 
the auxiliary multivector field
  $$
    F'=T(d-ikT^-)F.
  $$
  We note that $(d_T-ikT^-)(g_D+g_N)=(d_T-ikT^-)g_D=0$ at $\partial\Omega$, since $g_D= -(iku^0+\nabla_T u^0)$ and $\dim\partial\Omega=2$.
  Arguing as above for the Maxwell constraint, but swapping $T^+$ and $T^-$, 
  we verify that 
$F'$ solves the Dirac equation $\dirac F'=ikF'$, has homogeneous boundary conditions $N^+(F'|_{\partial\Omega})=0$ and satisfies the
  Dirac radiation condition $(\tfrac x{|x|}-1)F'=  o(|x|^{-1}e^{\im k |x|})$.
  
  We conclude that $F'=0$, and by Lemma~\ref{lem:constraint3dhelm}, 
  $\dirac F_n= ikF_n$, $n=1,2$, if $F_1= iku+\nabla u$ and $F_2=*\nabla v+ik{*v}$.
  Consequently 
  $$
    \Delta u +k^2 u =0= \Delta v+ k^2 v.
  $$
  This completes the derivation of Algorithm~\ref{alg:rot} for Helmholtz' equation.

\section{Relation to classical integral equations}   \label{sec:compare}

Classically, the Helmholtz scattering problem with Dirichlet data is solved using the double layer
potential operator
$$
   \dl_k u(x)= 2 \pv\int_{\partial\Omega} (\nabla\Phi_k)(y-x)\cdot n(y) u(y) dy,\qquad x\in\partial\Omega,
$$
and to avoid spurious interior resonances one modifies the integral equation $h-\dl_k h= g$ by adding a
single layer potential.
Similarly, the Neumann problem is solved using the adjoint of the double layer potential, that is the
normal derivative of the single layer potential
$$
   \dl_k^* u(x)= 2 \pv\int_{\partial\Omega} n(x)\cdot(\nabla\Phi_k)(x-y) u(y) dy,\qquad x\in\partial\Omega,
$$
and one avoids spurious interior resonances by adding a normal derivative
of the double layer potential. See Colton and Kress~\cite[Sec. 3.6]{CK0}.

Finally for the Maxwell scattering problem, one classically solves this using the magnetic dipole operator
$$
  M_k v(x)=2 \pv\int_{\partial\Omega} n(x)\times \big((\nabla\Phi_k)(x-y)\times v(y)\big) dy,\qquad x\in\partial\Omega,
$$
acting on tangential vector fields $v$,
and to avoid spurious interior resonances one can add a suitable electric dipole field.
See \cite[Thm. 6.18]{CK}.

In this section, we compare our spin integral equation to these classical integral equations.
It is not clear to the author exactly how the equations which avoid spurious interior resonances compare, in our case
the spin integral equation, and in the classical case the various combined field integral equations.
Instead we limit ourselves to comparing the simpler formulations which suffer from spurious interior resonances: In the classical 
case the operators $\dl_k$, $\dl_k^*$ and $M_k$, and in our case the rotation operator $C_kN$ as discussed below.
To see the relation, we return to Section~\ref{sec:Nbvp} and consider the more straightforward version 
$$
  (N^+C_k^--N^-C_k^+) \Rd h \Bk=g
$$
of the
spin integral equation, replacing $S^-$ by $N^-$, which we rewrite
$$
  (I-C_kN)(N \Rd h \Bk)=2g.
$$
This operator $C_kN$ was called the rotation operator in \cite{Ax1}.
It is seen that for $g\in N^+\mH$, this equation is always solvable, but at interior resonances 
$k\in\R$, $h$ may be unique only modulo a finite dimensional subspace of $\mH$.
However, any such solution $h$ will yield the same solution $F$ to the Dirac scattering problem.

\Rd
The following Proposition~\ref{prop:comprdlp} shows how the three classical integral operators above correspond to the compression $N^+C_k N^+$ of $C_k$ to
the subspace $N^+\mH$.
Note in particular that $I-N^+C_kN^+$ is invertible if and only if $I-\dl_k$, $I-M_k$ and
$I+\dl^*$ are all three invertible.

\begin{prop}   \label{prop:comprdlp}
   Write a tangential multivector field $f\in N^+ \mH$ as
   $f= u_1+(n\times v)+*(u_2n)+*0$, where $u_1, u_2$ are scalar functions and $v$ is a tangential vector field.
   Then in this splitting we have
   $$
     N^+C_kN^+=
     \begin{bmatrix}
        \dl_k & 0 & 0 & 0  \\
        \dl_k' & M_k & 0 &  0\\
        \dl_k'' & M_k' & -\dl_k^* & 0 \\
        0 & 0 & 0 & 0,
     \end{bmatrix}
   $$
   where $\dl_k$ and $M_k$ are the double layer potential operator and the magnetic dipole
   operator above, and
   \begin{align*} 
     \dl_k' u_1(x)&= 2ik   \int_{\partial\Omega} \Phi_k(y-x) n(x)\times n(y) u_1(y) dy, \\
     \dl_k'' u_1(x)&= 2\int_{\partial\Omega}  n(x)\cdot (\nabla \Phi_k(y-x)\times n(y)) u_1(y) dy,\\
     M_k' v(x)&= -2ik\int_{\partial\Omega} \Phi_k(y-x) n(x)\cdot v(y) dy.
   \end{align*}
\end{prop}

The proof is a straightforward examination of the Clifford products and exterior products used in the definition
of $C_k$ and $N^+$, which we leave to the reader.

The relation between the three classical integral operators above, which we have seen amount to the  compression $N^+C_k N^+$, and 
the rotation operator $C_kN$, which is closely related to our spin integral equation, is now as follows. 
Introducing the cosine operator
$$\tfrac 12((C_kN)+ (C_kN)^{-1}),$$
we see that
$\lambda\mapsto \tfrac 12(\lambda+1/\lambda)$ maps the spectrum of $C_kN$ 
onto the spectrum of this cosine operator.
This in turn is the direct sum of
$N^+C_kN^+$ and $-N^-C_kN^-$, and these two compressions are similar operators,
intertwined by the Hodge star map $*$.
For more details, we refer the reader to \cite{Ax1}.
\Bk

\bibliographystyle{acm}

\begin{thebibliography}{10}

\bibitem{Ax2}
{\sc Axelsson, A.}
\newblock Oblique and normal transmission problems for {D}irac operators with
  strongly {L}ipschitz interfaces.
\newblock {\em Comm. Partial Differential Equations 28}, 11-12 (2003),
  1911--1941.

\bibitem{Ax}
{\sc Axelsson, A.}
\newblock {\em Transmission problems for {D}irac's and {M}axwell's equations
  with {L}ipschitz interfaces.}
\newblock PhD thesis, The Australian National University, 2003.

\bibitem{Ax1}
{\sc Axelsson, A.}
\newblock Transmission problems and boundary operator algebras.
\newblock {\em Integral Equations Operator Theory 50}, 2 (2004), 147--164.

\bibitem{Ax3}
{\sc Axelsson, A.}
\newblock Transmission problems for {M}axwell's equations with weakly
  {L}ipschitz interfaces.
\newblock {\em Math. Methods Appl. Sci. 29}, 6 (2006), 665--714.

\bibitem{AGHMc}
{\sc Axelsson, A., Grognard, R., Hogan, J., and McIntosh, A.}
\newblock Harmonic analysis of {D}irac operators on {L}ipschitz domains.
\newblock In {\em Clifford analysis and its applications (Prague, 2000)},
  vol.~25 of {\em NATO Sci. Ser. II Math. Phys. Chem.} Kluwer Acad. Publ.,
  Dordrecht, 2001, pp.~231--246.

\bibitem{AMc}
{\sc Axelsson, A., and McIntosh, A.}
\newblock Hodge decompositions on weakly {L}ipschitz domains.
\newblock In {\em Advances in analysis and geometry}, Trends Math. Birkhäuser,
  Basel, 2004, pp.~3--29.

\bibitem{CK0}
{\sc Colton, D., and Kress, R.}
\newblock {\em Integral equation methods in scattering theory}, first~ed.
\newblock John Wiley \& Sons, New York, 1983.

\bibitem{CK}
{\sc Colton, D., and Kress, R.}
\newblock {\em Inverse acoustic and electromagnetic scattering theory},
  second~ed.
\newblock Springer-Verlag, Berlin, 1998.

\bibitem{EG1}
{\sc Epstein, C., and Greengard, L.}
\newblock Debye sources and the numerical solution of the time harmonic
  {M}axwell equations.
\newblock {\em Comm. Pure Appl. Math. 63}, 4 (2010), 413--463.

\bibitem{McMi}
{\sc McIntosh, A., and Mitrea, M.}
\newblock Clifford algebras and {M}axwell's equations in {L}ipschitz domains.
\newblock {\em Math. Methods Appl. Sci. 22}, 18 (1999), 1599--1620.

\bibitem{Tay2}
{\sc Taylor, M.}
\newblock {\em Partial differential equations. {II}. {Q}ualitative studies of
  linear equations}.
\newblock Springer-Verlag, New York, 1996.

\end{thebibliography}

\end{document}